\documentclass[a4paper,draft,reqno,12pt]{amsart}
\usepackage[english]{babel}
\usepackage{amsmath}
\usepackage{amssymb}
\usepackage{amscd}
\usepackage{amsthm}
\usepackage{euscript}
\newtheorem{prop}{Proposition}
\newtheorem{lem}{Lemma}
\newtheorem{theor}{Theorem}

\theoremstyle{definition}
\newtheorem{de}{Definition}

\theoremstyle{remark}
\newtheorem {re}{Remark}

\DeclareMathOperator{\cone}{cone}

\def\BG{{\mathbb G}}

\def\BG{{\mathbb G}}

\def\BC{{\mathbb C}}
\def\BK{{\mathbb K}}

\def\BZ{{\mathbb Z}}

\def\BN{{\mathbb N}}
\def\BQ{{\mathbb Q}}
\def\BP{{\mathbb P}}

\begin{document}
\date{}
\title[Additive actions on toric projective hypersurfaces]{Additive actions on toric projective hypersurfaces}
\author{Anton Shafarevich}
\address{Lomonosov Moscow State University, Faculty of Mechanics and Mathematics, Department of Higher Algebra, Leninskie Gory 1, Moscow, 119991 Russia; \linebreak and \linebreak
National Research University Higher School of Economics, Faculty of Computer Science, Pokrovsky Boulevard 11, Moscow, 109028 Russia}
\email{shafarevich.a@gmail.ru}

\thanks{The author was supported by the grant RSF-19-11-00172.}

\maketitle

\begin{abstract}
Let $\BK$ be an algebraically closed field of characteristic zero and $\BG_a$ be the additive group of $\BK$. We say that an irreducible algebraic variety $X$ of dimension $n$ over the field $\BK$ admits an  additive action if there is a regular action of the group $\BG_a^n = \BG_a \times \ldots \times \BG_a$ ($n$ times) on  $X$ with an open orbit. In this paper we find all projective toric hypersurfaces admitting additive action.
\end{abstract}

\section{Introduction}
Let $\BK$ be an algebraically closed field of characteristic zero and $\BG_a$ be the additive group of $\BK$. Consider an irreducible algebraic variety $X$ of dimension $n$ over $\BK$. Denote by $\BG_a^n$ the group $\BG_a \times \ldots \times \BG_a$ ($n$ times). \emph{An additive action} on $X$ is a regular action $\BG_a^n \times X \to X$ with an open orbit.

In \cite{HT}  Hassett and Tschinkel established a correspondence between additive actions on the projective space $\BP^n$ and local $(n+1)$-dimensional commutative associative algebras with unit. This correspondence allows to classify additive actions on $\BP^n$ when $n\leq 5$ and it shows that there are infinitely many additive actions on $\BP^n$ when $n>5$. This remarkable result motivated the study of additive actions on other algebraic varieties.

There are several results on additive actions on projective hypersurfaces \cite{AP, AS, SH}, flag varieties \cite{AF, DEV, FEI, FU}, singular del
Pezzo surfaces \cite{DER}, weighted projective planes \cite{ABZ} and general toric varieties \cite{[ARRO], DZ, DZ2}.

Denote by $\BK^{\times}$ the multiplicative group of the field $\BK$. Let $T$ be the torus $(\BK^{\times})^n$. Recall that a normal algebraic variety $X$ is called a \emph{toric} variety if $X$ contains $T$ as a Zariski open subset and the action of $T$ on itself extends to an action of $T$ on $X$. In \cite{[ARRO]} one can find a description of complete toric variety admitting an additive action. Also there was proven that complete toric variety $X$ admits an additive action if and only if $X$ admits an additive action which is normalized by the action of torus $T$ on $X$.

Our goal is to find all projective toric hypersurfaces admitting an additive action. Each projective toric variety can be given by a lattice polytope in the lattice of characters of a torus. An $n$-dimensional projective toric variety can be represented as a hypersurface in the projective space if and only if it can be given by a lattice polytope in the character lattice of the acting torus such that the number of lattice points inside this polytope is $n+2$. On the other hand, it is proven in \cite{[ARRO]} that a projective toric variety given by a lattice polytope admits an additive action if and only if this polytope is inscribed in a rectangle; see Definition \ref{ir} below. So in Proposition \ref{Polytopes} we find all lattice polytopes which meet these two conditions. It turns out that appart from the projective space there are two projective toric hypersurfaces in every dimension greater than 1 admitting an additive action (Theorem \ref{Mth}). They are quadrics of rank 3 and 4.    

In \cite{COX2} one can find a description of the automorphism group of a complete toric variety. In Section \ref{Aut} we compute the automorphism groups of quadrics of rank 3 and 4 using this reslut. 

There is a one-to-one correspondence between additive actions on projective hypersurfaces and pairs $(R, W)$ where $R$ is a local algebra and $W$ is a hypersurface in the maximal ideal of $R$ which generates $R$; see \cite{AP, AS}. In the last section we use this correspondence to find the number of additive actions on quadrics of rank 3 and 4 when dimension is equal to two, three or four. 

The author is grateful to Ivan Arzhantsev for the statement of the problem.

\section{Preliminaries}\label{prem}

Here we recall some basic facts on toric varieties. One can find all necessary information on toric varieties in \cite{COX, FUL}.
Also we give a discription of complete toric varieties admitting an additive action which was obtained in \cite{[ARRO]}.

\begin{de}
A \emph{toric variety} is a normal irreducible algebraic variety $X$ containing a torus $T$ as a Zariski open subset such that the action of $T$ on itself extends to a regular action of $T$ on $X$.
\end{de}

Let $M$ be the character lattice of $T$ and $N$ be the dual lattice of one-parameter subgroups of $T$. Denote by $M_{\BQ}$ the vector space $M\otimes_{\BZ}\BQ$ and by $N_{\BQ}$ the vector space $N\otimes_{\BZ}\BQ$. Denote by $\left<\cdot, \cdot\right>:N\times M \to \mathbb{Z}$ the natural pairing between lattices $N$ and $M$. This pairing can be extended to a map $N_{\BQ}\times M_{\BQ} \to \BQ$. 

\begin{de} A \emph{lattice polytope} $P\subseteq M_{\BQ}$ is the convex hull of a finite subset in $M$.
\end{de}

\begin{de}
Let $S\subseteq M$ be a semigroup. Then $S$ is \emph{saturated} if for all $k \in \BN\setminus \{0\}$ and for all $a\in S$ the condition $ka\in S$ implies $a\in S$.
\end{de}

Let $m$ be a vertex of a lattice polytope $P$. Denote by $S_{P, m}$ the semigroup in $M$ generated by the set $\left(P\cap M\right) - m$.

\begin{de}
A lattice polytope $P$ is \emph{very ample} if for every vertex $m \in P$, the semigroup $S_{P, m}$ is saturated.
\end{de}

Let $P$ be a full dimensional very ample lattice polytope and $P\cap M = \{m_1, \ldots, m_{s+1}\}$. Then one can consider the map

$$T \longrightarrow \BP^{s},\ \ \ t \longmapsto [\chi^{m_1(t)}: \ldots : \chi^{m_{s+1}(t)}],$$
where $\chi^{m_i}$ is the character of $T$ corresponding to $m_i$. Denote by $X_P$ the closure of the image of this map. Then $X_P$ is a projective toric variety.

Conversely, let $Y \subseteq \BP^{s}$ be a projective toric variety with acting torus $T$. Suppose that $Y$ is not contained in any hyperplane in $\BP^{s}$. Then $Y$ coincides with $X_P$ for some very ample polytope $P$ in $M$ up to automorphism of 
$\BP^{s}$.

\label{add}

Now let $X$ be an irreducible algebraic variety of dimension $n$. Denote by $\BG_a^n$ the group  $\BG_a \times \ldots \times \BG_a$ ($n$ times).

\begin{de}
An \emph{additive action} on a variety $X$ is a regular action $\BG_a^n \times X \to X$ with an open orbit.
\end{de}

Let $X$ be a complete toric variety with an acting torus $T$ and $M$ be the character lattice of $T$. Let $P$ be a lattice polytope in $M_{\BQ}$. Then each facet of $P$ is the intersection of $P$ with the hypersurface given by an equation of the from $\left<p, x\right> = a,$ where $p \in N, a \in \mathbb{Z}$ and $\left<p, x\right>\leq a$ for all $x$ in $P$. 

\begin{de}\label{ir}
Polytope $P$ is \emph{inscribed in a rectangle} if there is a vertex $v_0 \in P$ such that
\begin{itemize}
\item[(1)] the primitive vectors on the edges of $P$ containing $v_0$ form a basis $e_1, \ldots, e_n$ of the lattice $M$;

\item[(2)] for every inequality $\left<p, x\right>\leq a$ on $P$ that corresponds to a facet of $P$ not passing through $v_0$ we have $\left<p, e_i\right>\geq 0$ for all $i = 1, \ldots, n$.
\end{itemize}

The following theorem describes complete toric varieties admitting an additive action. 

\begin{theor} \cite[Theorem 4]{[ARRO]}
Let $P$ be a very ample polytope and $X_P$ be the corresponding projective toric variety. Then $X_P$ admits an additive action if and only if the polytope $P$ is inscribed in a rectangle.
\end{theor}

\end{de}

\section{Projective toric hypersurfaces admitting additive action}

The goal of this section is to describe projective toric hypersurfaces admitting an additive action.

Any hypersurface in a projective space is either a hyperplane or is not contained in any hyperplane. So any toric hypersurface is either hyperplane, and so isomorphic to a projective space, or it can be given by a lattice polytope as described above. It is well known that a projective space is a toric variety and it admits an additive action. So we will study only the second case.

Let $T$ be a torus and $M$ be the lattice of characters of $T$. Consider a full dimension very ample lattice polytope $P \subseteq M_{\BQ}$. Let $P\cap M = \{m_1, \ldots, m_{s+1}\}$ and  $X_P \subseteq \BP^{s}$ be the corresponding projective toric variety. Denote by $n$ the dimension of $X$ which is equal to dimension of $T$ and to rank of the lattice $M$. Since the only projective toric variety of dimension 1 is the projective line we will study only the case $n \geq 2$. Suppose that $X_P$ is a hypersurface in $\BP^{s}$. Then we have $n = s - 1$. So the polytope $P$ contains $n+2$ lattice points. 

Suppose $X_P$ admits an additive action. Then $P$ is inscribed in a rectangle. 

\begin{prop} \label{Polytopes} Let $M$ be the lattice of dimension $n$ with $n\geq 2$. Then up to translation and choosing basis of $M$ there are only two full dimensional lattice polytopes which are inscribed in rectangle and have $n+2$ lattice points. They are
$$A(n) = \mathrm{conv} \left(0, 2e_1, e_2, \ldots, e_n\right)$$
and
$$B(n) = \mathrm{conv}\left(0, e_1, \ldots, e_n, e_1 + e_2\right).$$

where $e_1, \ldots, e_n$ is a basis of $M$.
\end{prop}

\begin{proof}

Let $P$ be a polytope which satisfies the conditions of Proposition \ref{Polytopes}. Denote by $v_0$ the vertex of $P$ and by $e_1 \ldots, e_n$ the basis of $M$ from Definition \ref{ir}. We can assume that $v_0$ is the origin of $M$. Then $P$ has vertices
$$v_0 = (0, 0, \ldots, 0),\  v_1 = (k_1, 0, \ldots, 0),\  v_2 = (0, k_2, \ldots, ),\ldots, \  v_n = (0, 0, \ldots, k_n),$$
where $k_1, \ldots k_n$ are positive integers (all coordinates are written in the basis $e_1, \ldots, e_n$).

It is well known that a polytope of dimension $n$ has at least $n+1$ vertices. So either $P$ has $n+1$ vertices and there is a lattice point in $P$ which is not a vertex or $P$ has $n+2$ vertices and there is no lattice point in $P$ except vertices.

Consider the first case. Then $P$ has no vertices except $v_0, \ldots, v_n$. Since the convex hull of points $(0, 0, \ldots, 0), (1, 0, \ldots, 0), \ldots (0, 0, \ldots, 1)$ contains only $n+1$ lattice points, there is $i$ such that $k_i > 1$. Then $k_i = 2$ and $k_j = 1$ for all $j \neq i$. We can assume that $k_1 = 2$. Then $P$ coincides with the polytope $A(n)$. This polytope is given by the inequalities 

$$
\begin{cases}\begin{aligned}
x_i & \geq 0,\ i=1,\ldots,n,\\
x_1 &+ 2x_2 + 2x_3 + \ldots + 2x_n \leq  2.\\
\end{aligned}\end{cases}
 $$

The vertex $(0, 0, \ldots, 0)$ satisfies conditions of Definition \ref{ir} and only lattice points $(0, 0, \ldots, 0), (1, 0, \ldots, 0), (0, 1, \ldots, 0), \ldots, (0, 0, \ldots, 1), (0, 2, \ldots 0)$ belong to $A(n)$. So $A(n)$ is a lattice polytope inscribed in rectangle with $n+2$ lattice points. 

Consider the second case. Here $k_1 = k_2 = \ldots = k_n = 1$ and $P$ has vertices $v_0, v_1, \ldots, v_n$ and some vertex $b = (b_1, \ldots, b_n)$. Since $P$ is inscribed in rectangle, the coordinates $b_i$ are non-negative integers. There is a facet $F$ containing $b$ and $n-1$ vertices among $v_1, \ldots, v_n$. We can assume that these vertices are $v_2, \ldots, v_n$. Consider the corresponding inequality $\left<p, x\right> \leq a$, where $p\in N$ and $a$ is a positive integer. Let $p = \alpha_1 e^1 + \ldots + \alpha_n e^n$ where $\alpha_i$ are non negative rational numbers and $e^1, \ldots e^n$ is the basis dual to $e_1,\ldots, e_n$. We have
$$\left<p, v_2\right> = \ldots = \left<p, v_n\right> = a.$$
This implies 
$$\alpha_2 = \ldots = \alpha_n = a.$$
So we obtain
$$\left<p, b\right> = \alpha_1b_1 + ab_2 + \ldots + ab_n = a.$$

If $b_2 = \ldots = b_n = 0$ then $b=(b_1, 0, \ldots, 0)$ and either $b = v_1$ or $v_1$ is not a vertex. Both cases contradict our assumptions. So there is $i>1$ such that $b_i \neq 0$. We can assume that $b_2 \neq 0$. It is easy to see that it implies $b_2 = 1$ and $b_3 = \ldots = b_n = 0$. Then we have $b_1 \neq 0$ (in other case $b = v_2$). If $b_1>1$ then there is a lattice point $(1, 1, \ldots, 0)$ in $P$ except $b, v_0, \ldots, v_n$. So $b_1 = 1$.

Therefore $P$ coincides with $B(n)$. The polytope $B(n)$ is given by the following set of inequalities:

$$
\begin{cases}\begin{aligned}
x_i & \geq 0,\ i = 1,\ldots, n,\\
x_1 &+ x_3 + \ldots + x_n \leq  1,\\
x_2 &+ x_3 + \ldots + x_n \leq 1.
\end{aligned}\end{cases}
 $$

We see that the vertex $(0, 0, \ldots, 0)$ meets conditions from Definition \ref{ir} and no lattice point other than $(0, 0, \ldots, 0), (1, 0, \ldots, 0) \ldots, (0, 0, \ldots, 1)$ and  $(1,1,0,\ldots, 0)$ satisfies these inequalities. So $B(n)$ is also a lattice polytope inscribed in rectangle with $n+2$ lattice points. 

\end{proof}

It is left to show that polytopes $A(n)$ and $B(n)$ are very ample. 

\begin{prop}
Polytopes $A(n)$ and $B(n)$ are very ample.
\end{prop}

\begin{proof}
We recall that a lattice polytope $P\subseteq M$ is called \emph{normal} if

$$(kP)\cap M + (lP)\cap M = ((k+l)P)\cap M$$
for all $k,l \in \mathbb{N}$. Every normal polytope is very ample; see \cite[Proposition 2.2.17]{COX}.

We prove by induction by $n$ that polytopes $A(n)$ and $B(n)$ are normal. The base case is $n = 2$.  But every lattice polygon is normal  by \cite[Corollary 2.2.12]{COX}. 

Now we come to the inductive step. One can see that $$A(n+1) = \mathrm{conv} \left(A(n)\times\{0\}, \left(0, \ldots, 0, 1\right)\right)$$ and $$B(n+1) = \mathrm{conv} \left(B(n)\times\{0\}, \left(0, \ldots, 0, 1\right)\right).$$

Let us prove that if a lattice polytope $P\subseteq M$ is normal then the lattice polytope $$\overline{P} = \mathrm{conv} \left(P\times\{0\}, \left(0, \ldots, 0, 1\right)\right) \subseteq \overline{M} = M\times\mathbb{Z}$$ is normal as well. By \cite[Lemma 2.2.13]{COX} it is enough to prove that the set $\overline{P}\times\{1\}$ generates the semigroup $$\cone\left(\overline{P}\times\{1\}\right) \cap \overline{M}\times \mathbb{Z}.$$

Let $P\cap M = \{m_1, \ldots, m_{s + 1} \}.$ Then 
$$\overline{P}\cap \overline{M} = \{m_1\times \{0\}, \ldots, m_{s+1}\times \{0\}, (0, \ldots, 0, 1) \}.$$ 
We denote by $\overline{m_i}$ the point $m_i\times \{0\} \times \{1\}\in \overline{M}\times \mathbb{Z}$. Then $\overline{P}\times\{1\}$ contains only lattice points $\overline{m_1}, \ldots, \overline{m_{s + 1}}$ and  $(0, \ldots, 1, 1)$. Let $p$ be a lattice point in the cone generated $\overline{P}\times\{1\}$. Then 
$$p = \alpha_1\overline{m_1} + \ldots + \alpha_{s+1}\overline{m_{s+1}} + \alpha_{s+2}(0, \ldots, 1, 1)$$
where $\alpha_i$ are non-negative rational numbers. Only the last term has a non-zero penultimate coordinate.  So $\alpha_{s+2}$ should be an integer. Then $\alpha_1\overline{m_1} + \ldots + \alpha_{s+1}\overline{m_{s+1}}$ should be a lattice point. But due to \cite[Lemma 2.2.13]{COX} the points $\overline{m_1}, \ldots, \overline{m_{s+1}}$ generate all lattice points in the cone generated by them. Then $\alpha_1\overline{m_1} + \ldots + \alpha_{s+1}\overline{m_{s+1}}$ can be expressed as a linear combination of $\overline{m_1} \ldots, \overline{m_{s+1}}$ with non-negative integer coefficients and therefore $p$ belongs to the semigroup generated by $\overline{P}\times\{1\}$. 

\end{proof}

So $A(n)$ and $B(n)$ are very ample polytopes inscribed in rectangle with $n+2$ lattice points. Therefore they define toric hypersurface admitting an additive action. Denote this hypersurfaces by $X_{A(n)}$ and $X_{B(n)}$ respectively. 

The variety $X_{A(n)}$ is the closure of the image of the map $$\varphi_{A(n)}: T \to \mathbb{P}^{n+1},\ \varphi_{A(n)}(t_1, \ldots, t_n) = [1 : t_1: t_1^2: t_2: \ldots: t_n].$$

The image of this map is contained in the quadric $z_0z_2 = z_1^2$, where $z_0, \ldots, z_{n+1}$ are coordinates on $\mathbb{P}^{n+1}$. Since this quadric is irreducible the variety $X_{A(n)}$ coincides with this quadirc. 

The variety $X_{B(n)}$ is the closure of the image of the map $$\varphi_{B(n)}: T \to \mathbb{P}^{n+1},\ \varphi_{B(n)}(t_1, \ldots, t_n) = [1 : t_1: t_2: t_1t_2: \ldots: t_n].$$

Similarly, the variety $X_{B(n)}$ coincides with the quadric $z_0z_3 = z_1z_2$. 

As a corollary, we obtain the main theorem. We denote by $Q(n, k)$ the quadric of dimension $n$ and rank $k$ in $\mathbb{P}^{n+1}$. 

\begin{theor}\label{Mth} Let $X$ be a projective toric hypersurface of dimension $n$ admitting an additive action. Then \begin{enumerate}

\item if $n = 1$ then $X \simeq \mathbb{P}^1$;

\item if $n \geq 2$ then $X$ is isomorphic to $\mathbb{P}^n$, $Q(n, 3)$ or $Q(n, 4)$.
\end{enumerate} 
\end{theor}

\begin{re} 
If a complete toric variety $X$ admits an additive action then there is an additive action which is normalized by the acting torus, and any two normalized additive actions on $X$ are isomorphic; see \cite[Theorem 3]{[ARRO]}.

The normalized additive action on $Q(n, 3)$ has the form

$$(a_1, \ldots, a_n) \circ [z_0: z_1: z_2 : z_3: \ldots: z_{n+1}] = $$
$$=  [z_0: z_1 + a_1z_0:z_2 + 2a_1z_1 + a_1^2z_0: z_3 + a_2z_0: \ldots :z_{n+1} + a_nz_0],$$
where $(a_1, a_2, \ldots, a_n) \in \mathbb{G}_a^n$.

In turn, the normalized additive action on $Q(n, 4)$ has the form

$$(a_1, \ldots, a_n) \circ [z_0: z_1: z_2 : z_3: z_4: \ldots: z_{n+1}] = $$
$$= [z_0: z_1 + a_2z_0:z_2 + a_1z_0:z_3 + a_1z_1 + a_2z_2 + a_1a_2z_0: z_4 + a_3z_0: \ldots :z_{n+1} + a_nz_0].$$

\end{re}

\section{Automorphism groups of quadrics of rank 3 and 4}\label{Aut}

In this section we describe the automorphism group of quadrics of rank 3 and 4.  With any toric variety $X$ one can associate a fan $\Sigma = \Sigma_X$ of rational polyhedral cones in $N_{\mathbb{Q}}$; see \cite{COX, FUL}. Let $\Sigma_X(1) =  \{\rho_1, \ldots, \rho_m\}$ be the set of one-dimensional cones in $\Sigma_X$.  Each $\rho_i$ corresponds to a prime $T$-invariant Weil divisor $D_{i}$ in $X$. We denote by $[D_{i}]$ the class of $D_{i}$ in the divisor class group  $\mathrm{Cl} (X)$. One can find detailed information on the correspondence $\rho_i \to D_{i}$ and on the divisor class group of toric varieties in \cite[Chapter 4]{COX}.

The \emph{Cox ring} $R(X)$ of a toric variety $X$ is a polynomial algebra $\mathbb{K}[x_1, \ldots, x_m]$ with $\mathrm{Cl}(X)$-grading given by $\mathrm{deg}\ x_i = [D_{\rho_i}]$; see \cite[\S 1]{COX2} and \cite[Chapter 3]{ADHL} for a general definition of Cox ring of a normal variety with finitely generated class group. 

The Cox ring $R(X) = \mathbb{K}[x_1, \ldots, x_m]$ is a direct sum

$$\bigoplus_{a\in \mathrm{Cl}(X)} S_a,$$
 where each $S_a$ is a linear span of monomials of degree $a$. Denote by $\mathrm{Aut}_g (R(X))$ the subgroup in the group of automorphisms of $R(X)$ which consits of automorphisms $\psi$ such that $\psi(S_a) = S_a$ for all $a \in \mathrm{Cl} (X)$. Let $H_X$ be the group $\mathrm{Hom} (\mathrm{Cl} (X), \mathbb{K^*})$. Then $H_X$ is a direct product of a torus and a finite abelian group. The grading of $R(X)$ by $\mathrm{Cl}(X)$ defines an action of $H_X$ on $R(X)$. This action induces an embedding of $H_X$ into $\mathrm{Aut}_g (R(X))$ and $H_X$ is a central subgroup of $\mathrm{Aut}_g (R(X))$ with respect to this embedding.

Now suppose that $X$ is a complete toric variety. Let $\mathrm{Aut} (X)$  be the automorphism group of $X$ and $\mathrm{Aut}^0 (X)$ be the connected component of identity of $\mathrm{Aut} (X)$. It follows from \cite[Chapter 4]{COX2} that $\mathrm{Aut}_g (R(X))$ is an affine algebraic group and the group $\mathrm{Aut}_g (R(X))/H_X$ is isomorphic to $\mathrm{Aut}^0 (X)$.  

Each $S_{a}$ can be written as $S_a = S_a'\bigoplus S_a''$ where $S_a'$ is spanned by variables of degree $a$ and $S_a''$ is spanned by other monomials of degree $a$. Consider the subgroup $G_s \subseteq \mathrm{Aut}_g (R(X))$ which consists of non-degenerate linear transformations of spaces $S_a'$. Then $G_s$ is a reductive group isomorphic to 
$$\prod\limits_{a} \mathrm{GL} (S_a').$$

Now consider automorphisms in $\mathrm{Aut}_g (R(X))$ which adds to each variable $x_i$ some element of $S_a''$ where $a$ is the degree of $x_i$. Such automorphisms form an unipotent subgroup $R_u~\subseteq~\mathrm{Aut}_g (R(X))$ and the group $\mathrm{Aut}_g (R(X))$ is the semidirect product $G_s \ltimes R_u.$ See \cite[Chapter 4]{COX2} for details. 

Now we can apply these facts to compute the automorphism groups of quadrics of rank 3 and 4. To smiplify notation, we denote a projective quadric of rank $k$ as~$Q_k$.

\begin{theor}\label{propaut} 
\begin{enumerate}
\item The group $\mathrm{Aut}^0 (Q_3)$ is isomorphic to the quotient group of the group 
$$(\mathrm{GL}_2(\mathbb{K}) \times \mathrm{GL}_{n-1} (\mathbb{K})) \ltimes \mathbb{G}_a^{3(n-1)}$$
by the subgroup consisting of the following triples 

$$\left(\left(\begin{matrix} 
t & 0 \\
0 & t
\end{matrix}\right), 
\left(\begin{matrix}
t^2 & 0  & \ldots & 0\\
0 & t^2 &\ldots & 0\\
\vdots& \vdots& &\vdots \\
0 & 0 & \ldots & t^2\end{matrix}\right), 0
\right),\ t \in \mathbb{K}^*.$$

\item The group $\mathrm{Aut}^0 (Q_4)$ is isomorphic to the quotient group of the group  
$$(\mathrm{GL}_2(\mathbb{K}) \times \mathrm{GL}_{2} (\mathbb{K}) \times \mathrm{GL}_{n-2}(\mathbb{K})) \ltimes \mathbb{G}_a^{4(n-2)}$$
by the subgroup consisting of the following quadruples 

$$\left(\left(\begin{matrix} 
t_1 & 0 \\
0 & t_1
\end{matrix}\right), 
\left(\begin{matrix} 
t_2 & 0 \\
0 & t_2
\end{matrix}\right),
\left(\begin{matrix}
t_1t_2 & 0  & \ldots & 0\\
0 & t_1t_2 &\ldots & 0\\
\vdots& \vdots& &\vdots \\
0 & 0 & \ldots & t_1t_2\end{matrix}\right), 0
\right),\ t_1, t_2 \in \mathbb{K}^*.$$

\end{enumerate}
\end{theor}

\begin{proof}

Since for a complete toric variety $\mathrm{Aut}^0 (X)$ is isomorphic to the quotient group $\mathrm{Aut}_g (R(X))/H_X$ it is enough to find $\mathrm{Aut}_g (R(X))$ and $H_X$.

The set $\Sigma_{Q_3}(1)$ consits of rays which are orthogonal to faces of the polytope $A(n)$. The faces of this polytope are given as an intersection of hypersurfaces $x_i = 0$ where $i = 1,\ldots n$ and $x_1 + 2x_2 + \ldots + 2x_n = 2$ with the polytope $A(n)$. So we obtain 
$$\rho_1 = (1, 0, \ldots, 0),\ \rho_2 = (0, 1, \ldots, 0), \ldots, \rho_{n} = (0, 0, \ldots, 1), \rho_{n+1} = (-1, -2, \ldots, -2).$$

Then $R(Q_3)$ is the algebra $\mathbb{K}[x_1, \ldots, x_{n+1}]$. Direct calculations show that $\mathrm{Cl} (Q_3) = \mathbb{Z}$ and we have $\mathrm{deg}\ x_1 = \mathrm{deg}\ x_{n+1} = 1,$ $\mathrm{deg}\ x_2 = \ldots = \mathrm{deg}\ x_n = 2.$ So a space $S_a'$ is nonzero only if $a = 1$ or $a = 2$. The dimension of the space $S_1'$ is equal to $2$ and the dimension of the space $S_2'$ is equal to $n-1$. Therefore, the group $G_s$ is isomorphic to  $\mathrm{GL}_2(\mathbb{K}) \times \mathrm{GL}_{n-1} (\mathbb{K}).$ The group $R_u$ consists of automorphisms which adds to variables $x_2, \ldots, x_n$ a homogeneous polynomial of degree 2 in variables $x_1$ and $x_{n+1}$. This group is isomorphic to $\mathbb{G}_a^{3(n-1)}$. So  $\mathrm{Aut}_g (R(Q_3)) $ is isomorphic to

$$(\mathrm{GL}_2(\mathbb{K}) \times \mathrm{GL}_{n-1} (\mathbb{K})) \ltimes \mathbb{G}_a^{3(n-1)}.$$

Since $\mathrm{Cl} (Q_3)$ is isomorphic to $\mathbb{Z}$ then the group $H_{Q_3}= \mathrm{Hom} (\mathrm{Cl} (Q_3), \mathbb{K^*})$ is a one-dimensional torus $\mathbb{K}^*$. An element $t \in H_{Q_3}$ acts on variables $x_1$ and $x_{n+1}$ by multiplication by $t$ and on variables $x_2, \ldots, x_n$ by multiplication on $t^2$. So $H_{Q_3}$ is a subgroup of $(\mathrm{GL}_2(\mathbb{K}) \times \mathrm{GL}_{n-1} (\mathbb{K})) \ltimes \mathbb{G}_a^{3(n-1)}$ which consits of the following triples 

$$\left(\left(\begin{matrix} 
t & 0 \\
0 & t
\end{matrix}\right), 
\left(\begin{matrix}
t^2 & 0  & \ldots & 0\\
0 & t^2 &\ldots & 0\\
\vdots& \vdots& &\vdots \\
0 & 0 & \ldots & t^2\end{matrix}\right), 0
\right),\ t \in \mathbb{K}^*.$$

Now consider $Q_4$.The faces of $B(n)$ are given as an intersection of $B(n)$ with hypersurfaces $x_i = 0$ where $i = 1,\ldots n,$ and hypersurfaces $x_1 + x_3 + \ldots + x_n = 1$ and $x_2 + x_3 + \ldots + x_n = 1$. 
So the set $\Sigma_{Q_4} (1)$ consists of vectors 
$$\rho_1 = (1, 0, \ldots, 0),\ \rho_2 = (0, 1, \ldots, 0), \ldots \rho_n = (0, 0, \ldots, 1),$$ 
$$\rho_{n+1} = (-1, 0, -1, \ldots, -1),\ \rho_{n+2} = (0, -1,  -1, \ldots, -1).$$

Then $R(Q_4)$ is the algebra $\mathbb{K}[x_1, \ldots, x_{n+2}]$ and $\mathrm{Cl} (Q_4) = \mathbb{Z}\oplus \mathbb{Z}$. Up to automorphism of $\mathbb{Z}\oplus \mathbb{Z}$ one can assume that $\mathrm{deg} x_1 = \mathrm{deg} x_{n+1} = (1, 0), \mathrm{deg}\ x_2 = \mathrm{deg}\ x_{n+2} = (0, 1),$ and 
$$\mathrm{deg} x_3 = \mathrm{deg} x_4 = \ldots \mathrm{deg} x_n = (1, 1).$$

So a space $S_a'$ is nonzero only if $a$ is equal to $(1, 0), (0,1)$ or $(1,1)$. The dimension of $S_{(1,0)}'$ and $S_{(0,1)}'$ is $2$ and the  dimension of $S_{(1,1)}'$ is $n-2$. So the group $G_s$ is isomorphic to 
$$\mathrm{GL}_2(\mathbb{K}) \times \mathrm{GL}_{2} (\mathbb{K}) \times \mathrm{GL}_{n-2}(\mathbb{K}).$$

The group $R_u$ consists of automorphisms which adds to variables $x_3, \ldots, x_n$ a linear combination of monomials $x_1x_2, x_1x_{n+2}, x_{n+1}x_2$ and $x_{n+1}x_{n+2}$. This group is isomorphic to~$\mathbb{G}_a^{4(n-2)}.$ Then the group $\mathrm{Aut}_g (R(Q_4))$ is isomorphic to

$$(\mathrm{GL}_2(\mathbb{K}) \times \mathrm{GL}_{2} (\mathbb{K}) \times \mathrm{GL}_{n-2}(\mathbb{K})) \ltimes \mathbb{G}_a^{4(n-2)}.$$

Since the group $H_{Q_4}$ is a two-dimensional torus $(\mathbb{K}^*)^2$  and an element $(t_1, t_2) \in H_{Q_4}$ acts on variables $x_1$ and $x_{n+1}$ by multiplication by $t_1$, on variables $x_2$ and $x_{n+2}$ by multiplication by $t_2$ and on variables $x_3, \ldots, x_n$ by multiplication by $t_1t_2$. Therefore the subgroup $H_{Q_4}$  consists of the following quadruples 

$$\left(\left(\begin{matrix} 
t_1 & 0 \\
0 & t_1
\end{matrix}\right), 
\left(\begin{matrix} 
t_2 & 0 \\
0 & t_2
\end{matrix}\right),
\left(\begin{matrix}
t_1t_2 & 0  & \ldots & 0\\
0 & t_1t_2 &\ldots & 0\\
\vdots& \vdots& &\vdots \\
0 & 0 & \ldots & t_1t_2\end{matrix}\right), 0
\right),\ t_1, t_2 \in \mathbb{K}^*.$$

\end{proof}

To complete the computations of the automorphism groups of $Q_3$ and $Q_4$ it is left to find the quotient groups $\mathrm{Aut}(Q_3)/\mathrm{Aut}^0(Q_3)$ and $\mathrm{Aut}(Q_4)/\mathrm{Aut}^0(Q_4)$.

\begin{prop}
The quotient group $\mathrm{Aut}(Q_3)/\mathrm{Aut}^0(Q_3)$ is trivial. The quotient group $\mathrm{Aut}(Q_4)/\mathrm{Aut}^0(Q_4)$ is isomorphic to $\mathbb{Z}/2\mathbb{Z}$.
\end{prop}

\begin{proof}
For a complete toric variety $X$ with a fan $\Sigma$ we denote by $\mathrm{Aut} (N, \Sigma)$ the group of  automormisms  of the lattice $N$ which preserve the fan $\Sigma$ and by $\mathrm{Aut}^0 (N, \Sigma)$ the subgroup 

$$\mathrm{Aut}^0 (N, \Sigma) = \left\{\varphi \in \mathrm{Aut} (N, \Sigma)|\ [D_{\varphi(\rho)}] = [D_{\rho}] \  \text{for all }\rho \in \Sigma(1) \right\}.$$

It follows from \cite[Corollary 4.7]{COX2} that the group $\mathrm{Aut}(X)/\mathrm{Aut}^0(X) $ is isomorphic to $\mathrm{Aut} (N, \Sigma)/\mathrm{Aut}^0 (N, \Sigma)$. 

The divisor class group of $Q_3$ is isomorphic to $\mathbb{Z}$. The set $\Sigma_{Q_3}(1)$ is described in the proof of Proposition \ref{propaut}. Then up to automorphism of $\mathrm{Cl} (Q_3)$ we can assume that $[D_{\rho_1}] = [D_{\rho_{n+1}}] =~1$ and $[D_{\rho_2}] = \ldots = [D_{\rho_n}] = 2$. Since every automorphism $\varphi \in \mathrm{Aut} (N, \Sigma_{Q_3})$ induces an automorphism of $\mathrm{Cl}(Q_3)$, we have $[D_{\varphi(1)}] = [D_{\varphi(\rho_{n+1})}] = 1$. So  $\mathrm{Aut} (N, \Sigma_{Q_3}) = \mathrm{Aut}^0 (N, \Sigma_{Q_3})$ and the quotient group $\mathrm{Aut}(Q_3)/\mathrm{Aut}^0(Q_3)$ is trivial. 

The divisor class group of $X_{Q_4}$ is isomorphic to $\mathbb{Z} \oplus \mathbb{Z}$. Using notations from the proof of Proposition \ref{propaut} we can assume that 
$$[D_{\rho_1}] = [D_{\rho_{n+1}}] = (1, 0),\   [D_{\rho_2}] = [D_{\rho_{n+2}}] = (0, 1),\   [D_{\rho_3}] =  \ldots = [D_{\rho_n}] =  (1,1).$$

Consider an automorphism $\varphi \in \mathrm{Aut} (N, \Sigma_{Q_4})$. Then either we have the case $$[D_{\varphi(\rho_1)}] = [D_{\varphi(\rho_{n+1})}] = (1, 0),\ [D_{\varphi(\rho_2)}] = [D_{\varphi(\rho_{n+2})}] = (0, 1),$$ $$[D_{\varphi(\rho_3)}]  = \ldots = [D_{\varphi(\rho_n})] =  (1,1),$$ or we have the case $$[D_{\varphi(\rho_1)}] = [D_{\varphi(\rho_{n+1})}] = (0, 1),\ [D_{\varphi(\rho_2)}] = [D_{\varphi(\rho_{n+2})}] = (1, 0),$$ $$[D_{\varphi(\rho_3)}]  = \ldots = [D_{\varphi(\rho_n})] =  (1,1).$$

In the first case the automorphism $\varphi$ belongs to the group $\mathrm{Aut}^0 (N, \Sigma_{Q_4})$. The automorphism of $N$, which transposes the basis vectors $e_1$ and $e_2$ and preserves the basis vectors $e_3, \ldots, e_n$, is an example of an automorphism of the second type. Any two automorphisms of the second type differ by an automorphism of the first type. So the quotient groups $\mathrm{Aut}(N, \Sigma_{Q_4})/\mathrm{Aut}^0(N, \Sigma_{Q_4})$ and $\mathrm{Aut}(Q_4)/\mathrm{Aut}^0(Q_4)$ are isomorphic to $\mathbb{Z}/2\mathbb{Z}$.
\end{proof}

\section{Additive actions and local algebras}

In this section we find the number of additive actions on quadrics of rank 3 and 4 when dimension is equal to two, three or four. 

There is a connection between additive actions on projective hypersurfaces and local algebras which is described in \cite{AS, AP}. By a local algebra we mean a finite-dimensional commutative associative algebra with unit over the field $\mathbb{K}$.

Consider a local algebra $R$ of dimension $n+2$ with the maximal ideal $\mathfrak{m}$.  Consider an $n$-dimensional subspace $W \subseteq \mathfrak{m}$ that generates $R$ as an algebra with unit. Then the group $\mathrm{exp} (W)$ is isomorphic to the group $\mathbb{G}_a^n$ and acts on $R$ by multiplication. This action induces an action of $\mathbb{G}_a^n$ on the projective space $\mathbb{P} (R)$. If $a$ is a non-zero element of $R$ then by $\overline{a}$ we denote the image of $a$ in $\mathbb{P}(R)$. Closure of the orbit of $\overline{1}$ is a hypersurface and the group $\mathbb{G}_a^n$ acts with an open orbit on this hypersurface. Every additive action on a projective hypersurface of degree at least 2 can be obtained this way.

Let us formulate this statement more precisely. Let $(R_1, W_1)$ and $(R_2, W_2) $ be two pairs of local $(n+2)$-dimensional algebras $R_1$ and $R_2$ and $n$-dimensional subspaces $W_1$ and $W_2$ in the maximal ideals of these algebras that generate $R_1$ and $R_2$, respectively. We say that these pairs are  isomorphic if there is an isomorphism $\varphi: R_1 \to R_2$ of algebras such that $\varphi(W_1) = W_2$. 

We say that an additive action of the group $\mathbb{G}_a^n$ on a variety $X$ \emph{is equivalent} to an additive action on a variety $Y$ if there is $\mathbb{G}_a^n$-equivariant  isomorphism between $X$ and $Y$.
 
\begin{theor}\cite[Proposition 3]{AP} There is a one-to-one correspondence between \\

\begin{enumerate}

\item equivalence classes of additive actions on hypersurfaces in $\mathbb{P}^{n+1}$ of degree at least two;\\

\item isomorphism classes of pairs $(R, W)$, where $R$ is a local $(n+2)$-dimensional algebra with the maximal ideal $\mathfrak{m}$ and $W$ is a hyperplane in $\mathfrak{m}$ that generates $R$ as an algebra with unit. 

\end{enumerate}
\end{theor}

It is proven in \cite[Theorem 5.1]{AS} that the degree of the hypersurface corresponding to a pair $(R, W)$ is the maximal exponent $d$ such that subspace $W$ does not contain the ideal~$\mathfrak{m}^d$. 

\begin{prop}
\begin{enumerate}
\item There are two equivalence classes of additive actions on $Q(2, 3)$. 
\item There is a unique equivalence class of additive actions on $Q(2, 4)$. 
\end{enumerate}
\end{prop}
\begin{proof}
 
There are four isomorphism classes of $4$-dimensional local algebras; see \cite{HT}: \\
$$\begin{aligned}
&R_1 = \mathbb{K}[x]/(x^4),&\  &R_2 = \mathbb{K}[x_1, x_2]/(x_1^2, x_2^2),\\
&R_3 = \mathbb{K}[x_1, x_2]/(x_1^3, x_1x_2, x_2^2),&\  &R_4 = \mathbb{K}[x_1, x_2, x_3]/(x_i^2, x_ix_j).\\
\end{aligned}$$

\smallskip\emph{The case of} $R_1$. In the algebra $R_1$ the maximal ideal $\mathfrak{m}_1$ is  $\left<x, x^2, x^3\right>$. Since we are studying quadrics we are looking for $2$-dimensional subspace $W$ such that $W$ generates $R_1$ and $W$ contains $\mathfrak{m}^3$ but not $\mathfrak{m}^2$. It easy to see that up to automorphism of $R_1$ there is only one such space $W_1 = \left<x, x^3 \right>$. Consider an arbitrary element $a_1x + a_2x^3$ of subspace $W_1$. Then we have
$$\mathrm{exp} (a_1x + a_2x^3) = 1 + a_1x + \frac{a_1^2}{2}x^2 + (\frac{a_1}{6} + a_2)x^3.$$

The closure of the orbit of $\overline{1}$ in $\mathbb{P}(R_1)$ is a quadric of rank 3 given by the equation
$$2X_0X_2 - X_1^2 = 0,$$ 
where $X_0, X_1, X_2, X_3$ are homogeneous coordinates on $\mathbb{P}(R_1)$ which are induced by the basis $1, x, x^2, x^3$ of $R_1$.   So the pair $(R_1, W_1)$ corresponds to an additive action on $Q(2, 3)$.

\smallskip\emph{The case of} $R_2$. In this case the maximal ideal $\mathfrak{m}_2$ is $\left<x_1, x_2, x_1x_2 \right>$. Then up to automorphism of $R_2$ the subspace $W_2$ is  $\left< x_1, x_2 \right>$.  We have
$$\mathrm{exp}(a_1x_1 + a_2x_2) = 1 + a_1x_1 + a_2x_2 + a_1a_2x_1x_2.$$

The closure of the orbit of $\overline{1}$ in $\mathbb{P}(R_2)$ is a quadric of rank 4 given by the equation
$$X_0X_3 - X_1X_2 = 0.$$ 
So the pair $(R_2, W_2)$ corresponds to an additive action on $Q(2, 4)$.

\smallskip\emph{The case of} $R_3$. The maximal ideal $\mathfrak{m}_3$ is $\left<x_1, x_2, x_1^2 \right>$ and up to an automorphism of $R_3$ the only suitable subspace $W_3$ is  $\left<x_1, x_2 \right>$. In this case we have

$$\mathrm{exp}(a_1x_1 + a_2x_2) = 1 + a_1x_1 + \frac{a_1^2}{2}x_1^2 + a_2x_2.$$

Then the closure of the orbit of $\overline{1}$ in $\mathbb{P}(R_3)$ is a quadric of rank 3 given by the equation
$$X_0X_2 - \frac{1}{2}X_1^2 = 0.$$ 
 
So the pair $(R_3, W_3)$ corresponds to an additive action on $Q(2, 3)$.

\smallskip\emph{The case of} $R_4$. The maximal ideal $\mathfrak{m_4}$ is $\left<x_1, x_2, x_3 \right>$ and there is no suitable subspace~$W_4$. 

\end{proof}

\begin{re}
This result agrees with the recent classification of additive actions on complete toric surfaces given by Dzhunusov in \cite[Theorem 3]{DZ}. 
\end{re}

The case of quadrics $Q(3, 3)$ or $Q(3, 4)$ can be considered in a similar way.

\begin{prop}
\begin{enumerate}
\item There are seven equivalence classes of additive actions on $Q(3, 3)$;
\item There are three equivalence classes of additive actions on $Q(3, 4)$.
\end{enumerate}
\end{prop}

\begin{proof}
We need to classify pairs $(R, W),$ where $R$ is a local algebra of dimension 5 with the maximal ideal $\mathfrak{m}$ and $W$ is a hypersurface in the maximal ideal that generates $R$. We consider only subspaces $W$ such that $W$ contains $\mathfrak{m}^3$ but not $\mathfrak{m}^2$. In Table \ref{tab:metka} we list all such pairs. 

Each row corresponds to a desired pair $(R, W)$. So there are seven additive actions on a quadric of rank 3 and three additive actions on a quadric of rank 4.

\begin{table}
\begin{tabular}{ | c | c | c |}
\hline
 $R$ & $W$ & \ Type of hypersurface \\ \hline
$\mathbb{K}[x]/(x^5)$ &  $\left<x, x^3, x^4\right>$ &\ \ \ quadric of rank 3\\ \hline
$\mathbb{K}[x_1, x_2]/(x_1^3, x_2^3,x_1x_2)$ &  $\left<x_1, x_2, x_1^2\right>$ &\ \ \ quadric of rank 3\\ \hline
$\mathbb{K}[x_1, x_2]/(x_1^3, x_2^3,x_1x_2)$ &  $\left<x_1, x_2, x_1^2 + x_2^2\right>$ &\ \ \ quadric of rank 4\\ \hline
$\mathbb{K}[x_1, x_2, x_3]/(x_1x_2, x_1x_3, x_2x_3, x_1^2 - x_2^2, x_1^2 - x_3^2)$ &  $\left<x_1, x_2, x_3\right>$ &\ \ \ quadric of rank 5\\ \hline
$\mathbb{K}[x_1, x_2, x_3]/(x_1^2, x_1x_2, x_1x_3, x_2x_3, x_2^2 - x_3^2)$ & $\left<x_1, x_2, x_3\right>$ &\ \ \ quadric of rank 4\\ \hline
$\mathbb{K}[x_1, x_2]/(x_1x_2, x_1^3 - x_2^2)$ & $\left<x_1, x_2, x_1^3\right>$ &\ \ \ quadric of rank 3\\ \hline
$\mathbb{K}[x_1, x_2]/(x_1^4, x_2^2, x_1x_2)$ & $\left<x_1, x_2, x_1^3\right>$ &\ \ \ quadric of rank 3\\ \hline
$\mathbb{K}[x_1, x_2]/(x_1^4, x_2^2, x_1x_2)$ & $\left<x_1, x_2 + x_1^2, x_1^3\right>$ &\ \ \ quadric of rank 3\\ \hline
$\mathbb{K}[x_1, x_2]/(x_1^3 x_2^2, x_1^2x_2)$ & $\left<x_1, x_2, x_1^2\right>$ &\ \ \ quadric of rank 4\\ \hline
$\mathbb{K}[x_1, x_2]/(x_1^3, x_2^2, x_1^2x_2)$ & $\left<x_1^2, x_2, x_1x_2\right>$ &\ \ \ quadric of rank 3\\ \hline
$\mathbb{K}[x_1, x_2, x_3]/(x_1^3, x_2^2, x_3^2, x_1x_2, x_1x_3, x_2x_3)$ & $\left<x_1, x_2, x_3\right>$ &\ \ \ quadric of rank 3\\ 

\hline
\end{tabular}\vspace{0.1 cm}
\caption{}\label{tab:metka}

\end{table}

\end{proof}

\begin{prop}
There are infinitely many equivalence classes of additive actions on $Q(4, 4)$.
\end{prop}

\begin{proof}
Consider a local algebra $R = \mathbb{K}[x_1, x_2]/(x_1x_2, x_1^3 - x_2^3)$. There is a family of subspaces $W_{\alpha} = \left<x_1, x_2, x_1^2 + \alpha x_2^2, x_1^3 \right>$ in the maximal ideal $\mathfrak{m} = \left<x_1, x_2, x_1^2, x_2^2, x_1^3\right>, \alpha \in \mathbb{K}$. 

\begin{lem}
Let $\varphi: R \to R$ be an automorphism of $R$ such that $\varphi(W_{\alpha_1}) = W_{\alpha_2}$. Then either $\alpha_1 = \varepsilon \alpha_2$, or $\alpha_1\alpha_2 = \varepsilon$, where $\varepsilon$ is a root of unity of degree 3. 
\end{lem}

\begin{proof}
Let $a$ be an element of $\mathfrak{m}$. Then $a$ can be represented in the form
$$a = a_1x_1 + a_2x_2 + a_3x_1^2 + a_4x_2^2 + a_5x_1^3.$$ 
Denote by $\mathrm{Ann} (a)$ the annihilator of $a$. Direct calculations show that dimension of $\mathrm{Ann} (a)$ is equal to 2 if and only if $a_1a_2 \neq 0$. Dimension of $\mathrm{Ann} (x_1)$ is 3 so dimension of $\mathrm{Ann} (\varphi(x_1))$ is also equal to $3$. Since $\varphi(x_1) \in W_{\alpha_2}$, the element $\varphi(x_1) $ can be represented as 
$$\varphi(x_1) = b_1x_1 + b_2x_2 + b_3(x_1^2 + \alpha_2x_2^2) + b_4x_1^3.$$ 

But $\mathrm{dim\  Ann} (\varphi(x_1)) = 3$ so either $b_1 = 0$, or $b_2 = 0$.

Consider the case $b_1 = 0$.  The same arguments show that in this case $\varphi(x_2)$ can be represented as 
$$\varphi(x_2) = c_1x_1 + c_3(x_1^2 + \alpha_2x_2^2) + c_4x_1^3.$$

Let us apply the automorphism $\varphi$ to the element $x_1^2 + \alpha_1x_2^2$:
\begin{equation*} 
\begin{split}\varphi(x_1^2 + \alpha_1x_2^2) & =  \varphi(x_1)^2 + \alpha_1\varphi(x_2)^2 =\\
&= (b_2x_2 + b_3(x_1^2 + \alpha_2x_2^2) + b_4x_1^3)^2 + \alpha_1( c_1x_1 + c_3(x_1^2 + \alpha_2x_2^2) + c_4x_1^3)^2 = \\
&= b_2^2x_2^2 + 2b_2b_3\alpha_2x_1^3 + \alpha_1(c_1^2x_1^2 + 2c_1c_3x_1^3)=\\ 
&= b_2^2x_2^2 + \alpha_1c_1^2x_1^2 + (2b_2b_3\alpha_2 + 2\alpha_1c_1c_3)x_1^3. \end{split} \end{equation*}

This element belongs to $W_{\alpha_2}$ so $\frac{b_2^2}{\alpha_1c_1^2} = \alpha_2$. But since $x_1^3 = x_2^3$ we have
$$\varphi(x_1)^3 = b_2^3x_2^3 =  \varphi(x_2)^3 = c_1^3x_1^3.$$
Then $\frac{b_2^3}{c_2^3} = 1$ so $\frac{b_2}{c_2}$ is a root of unity of degree 3. Then we have $\alpha_1\alpha_2 = \varepsilon$.

Now consider the case $b_2 = 0$. Then $$\varphi(x_2) = c_2x_2 +  c_3(x_1^2 + \alpha_2x_2^2) + c_4x_1^3.$$

We have

\begin{equation*}
\begin{split}
\varphi(x_1^2 + \alpha_1x_2^2)  &=  \varphi(x_1)^2 + \alpha_1\varphi(x_2)^2 =\\
& = (b_1x_1 + b_3(x_1^2 + \alpha_2x_2^2) + b_4x_1^3)^2 + \alpha_1( c_2x_2 + c_3(x_1^2 + \alpha_2x_2^2) + c_4x_1^3)^2 = \\
&= b_1^2x_1^2 + 2b_2b_3\alpha_2x_1^3 + \alpha_1(c_2^2x_2 + 2c_2c_3x_1^3) = \\
&= b_2^2x_2^2 + \alpha_1c_2^2x_1^2 + (2b_2b_3\alpha_2 + \alpha_1c_2c_3)x_1^3.\end{split} \end{equation*}

So $\frac{\alpha_1c_2^2}{b_2^2} = \alpha_2$ and we obtain $\alpha_1 = \varepsilon\alpha_2$.

\end{proof}

So for every $\alpha \in \mathbb{K}$ the pair $(R, W_{\alpha})$ is isomorphic to a pair $(R, W_{\alpha'})$ only for a finite number of values $\alpha'$. 

It is left to check that every pair $(R, W_{\alpha})$ with $\alpha \neq 0$ corresponds to an additive action on a quadric of rank 4. 

It is easy to see that the space $W_{\alpha}$ generates $R$ and contains $\mathfrak{m}^3$ but not $\mathfrak{m}^2$. So we have

$$\mathrm{exp} (a_1x_1 + a_2x_2 + a_3(x_1^2 + \alpha x_2^2) + a_4x_1^3) = $$
$$ =  1 + a_1x_1 + a_2x_2 + (\frac{a_1^2}{2} + a_3)x_1^2 + (\frac{a_2^2}{2} + \alpha a_3)x_2^2 + (\frac{a_1^3}{6} + \frac{a_2^3}{6} + 2a_1a_3 + 2\alpha a_2a_3 + a_4)x_1^3. $$

The closure of the orbit of $\overline{1}$ in $\mathbb{P} (R)$ is a quadric given by
$$\alpha X_3X_0 - X_4X_0 - \frac{\alpha}{2}X_1^2 + \frac{X_2^2}{2},$$
where the homogeneous coordinates $X_0, X_1, X_2, X_3, X_4, X_5$ on $\mathbb{P} (R)$ are induced by the basis $1, x_1, x_2, x_1^2, x_2^2, x_1^3$ of $R$. When $\alpha \neq 0$ this is a quadric of rank 4.

\end{proof}

\begin{re}
One can check that there are finitely many equivalence classes of additive actions on $Q(4, 3)$. But there are 25 isomorphism classes of $6$-dimensional local algebras and some algebras have multiple suitable subspaces $W$. So it is hard to list them all.
\end{re}

\end{document}